\documentclass[11pt,a4paper]{amsart}
\usepackage[utf8]{inputenc}
\usepackage[T1]{fontenc}
\usepackage{xspace}

\usepackage[foot]{amsaddr}

\title{The metric Menger problem}
\author{Júlia Baligács$^1$ and Joseph MacManus$^2$}
\date{8th March, 2024}

\address{$^1$Fachbereich Mathematik, TU Darmstadt, Dolivostrasse 15, 64293 Darmstadt}
\address{$^2$Mathematical Institute,  University of Oxford, Oxford, OX2 6GG, UK}

\email{baligacs@mathematik.tu-darmstadt.de, macmanus@maths.ox.ac.uk}

\usepackage[a4paper,margin=1.25in]{geometry}

\usepackage{thmtools, thm-restate}
\usepackage{amsfonts}
\usepackage{amsmath}
\usepackage{amsthm}
\usepackage{amssymb}
\usepackage{stmaryrd}
\usepackage{mathrsfs}  

\usepackage{algorithm2e}

\usepackage{graphicx}
\usepackage[dvipsnames]{xcolor}
\usepackage{tikz}
\usepackage{tgpagella}
\usepackage[colorlinks]{hyperref} 
\hypersetup{
    linkcolor=black,
    citecolor=blue,
}
\usepackage{fancyref}

\usepackage{complexity}

\usepackage{todonotes}

\DeclareMathOperator{\dist}{d}

\DeclareMathOperator{\chck}{check}

\newlang{\MM}{MM}

\newcommand{\N}{\mathbb{N}}

\newcommand{\bigO}{\mathcal{O}}

\renewcommand{\subset}{\subseteq}

\newtheorem{theorem}{Theorem}[section]

\newtheorem*{theorem*}{Theorem}

\newtheorem{lemma}[theorem]{Lemma}

\newtheorem{question}[theorem]{Question}

\newtheorem{observation}[theorem]{Observation}

\theoremstyle{definition}
\newtheorem{definition}[theorem]{Definition}

\newtheorem*{definition*}{Definition}

\newtheorem*{problem*}{Problem}

\makeatletter
\newcommand{\myitem}[1]{%
\item[#1]\protected@edef\@currentlabel{#1}%
}
\makeatother

\begin{document}

\maketitle

    

\begin{abstract}
We study a generalization of the well-known disjoint paths problem which we call the \textit{metric Menger problem}, denoted $\MM(r,k)$, where one is given two subsets of a graph and must decide whether they can be connected by $k$ paths of pairwise distance at least~$r$. 
We prove that this problem is $\NP$-complete for every $r\geq 3$ and $k\geq 2$ by giving a reduction from $3\SAT$. This resolves a conjecture recently stated by Georgakopoulos and Papasoglu. 
On the other hand, we show that the problem is in $\XP$ when parameterised by treewidth and maximum degree by observing that it is `locally checkable'. In the case $r\leq 3$, we prove that it suffices to parameterise by treewidth. We also state some open questions relating to this work. 
\end{abstract}




\section*{Introduction}

    A classical theorem of Menger \cite{menger} states that the size of a smallest vertex cut separating two vertices $u$ and $v$ in an undirected graph $G$ is equal to the maximum number of pairwise internally disjoint paths connecting $u$ to $v$. As an important precursor to the celebrated max-flow min-cut theorem, Menger's theorem has found applications across graph theory and beyond. 

    Recently, a light has been shone upon Menger's theorem within the fast-growing area of `coarse graph theory'. To summarise, coarse graph theory (and more generally, coarse geometry) relates to metric properties which can be seen on a `large scale'. Much of the motivation for this area comes from the field of geometric group theory, where algebraic properties of finitely generated groups can often be deduced from the coarse geometry of its Cayley graphs. Key themes include, but are not limited to, finding coarse characterisations of certain classes of graphs \cite{fujiwara2023coarse,  kerr2023tree, macmanus2023accessibility}, computing coarse invariants such as the Assouad--Nagata dimension or asymptotic dimension of graphs \cite{bonamy2023asymptotic, esperet2023coarse, fujiwara2021asymptotic}, and proving coarse variants of classical results in graph theory \cite{albrechtsen2023induced, georgakopoulos2023graph}. 
    
    In particular, a coarse variant of Menger's theorem was conjectured independently by Albrechtsen et al.~\cite {albrechtsen2023induced} and Georgakopoulos--Papasoglu \cite{georgakopoulos2023graph}. Roughly speaking, this conjecture states that given a (possibly infinite) connected graph $G$, two subsets $A , Z \subset V(G)$, and $k \geq 2$, we should either find $k$ paths connecting $A$ to $Z$ which are pairwise `sufficiently far apart', or we should find $k$ balls of bounded radius in our graph, the removal of which will disconnect $A$ and $Z$. A counterexample to this conjecture was recently found in \cite{nguyen2024counterexample}, though several closely-related variants of this conjecture remain open  (see \cite{nguyen2024counterexample} for precise statements). 

    Thinking specifically about finite graphs, it makes sense to consider an algorithmic variant of the coarse Menger conjecture. Given two subgraphs $P$, $Q$ of a graph~$G$ and $m \geq 1$, we say that $P$ and $Q$ are \textit{$m$-disjoint} if every path connecting a vertex of $P$ to a vertex of $Q$ has length at least $m+1$. Note that two paths are \mbox{0-disjoint} if and only if they are vertex disjoint. We now define the following decision problem known as the metric-Menger problem, first introduced by Georgakopoulos--Papasoglu in \cite{georgakopoulos2023graph}.\footnote{More precisely, they define a slightly less general class of problems. In their notation, the problem `$\MM3$' they define corresponds to $\MM(3,3)$ in our notation.}

    \begin{problem*}
        Given $r, k \geq 1$, define $\MM(r, k)$ as follows:
        
            \textit{Instance:} Finite graph $G$, subsets $A, Z \subset V(G)$.
            
            \textit{Question:} Are there $k$ pairwise $(r-1)$-disjoint paths $P_1, \ldots P_k$ from $A$ to $Z$?
    \end{problem*}

When setting $r = 1$, it is easy to see that $\MM(1,k)$ is polynomial-time reducible to the classical disjoint paths problem (see Lemma~\ref{lem:mmp-to-mm}), in which the task is to connect $k$ given pairs of vertices by $k$ vertex-disjoint paths. One of the highlights of the Robertson--Seymour graph minors project was a cubic-time solution to the disjoint paths problem given we fix the number of desired paths \cite{robertson1995graph}. This was later improved to quadratic time in  \cite{kawarabayashi2012disjoint}. In particular, we easily see that $\MM(1,k)$ is in $\P$ for every $k \geq 2$. On the other hand, if $k$ is part of the input, the disjoint paths problem is $\NP$-complete. In fact, it is one of Karp's original NP-complete problems~\cite{Karp75}.

In contrast to $\MM(1,k)$ being in $\P$, Georgakopoulos--Papasoglu conjectured that \linebreak $\MM(3,3)$ is $\NP$-complete (see Conjecture~9.6 of \cite{georgakopoulos2023graph}). In this note, we answer this conjecture in the affirmative and go slightly further.

\begin{restatable}{alphtheorem}{npc}\label{thm:np-complete}
For every $r \geq 3$, $k \geq 2$, the problem $\MM(r,k)$ is $\NP$-complete, even on graphs of degree at most four.
\end{restatable}

To prove this, we construct a Karp reduction from $3\SAT$ to $\MM(r,k)$. 
The instance of $\MM(r,k)$ we construct consists of a graph $G$ satisfying $\Delta(G) = 4$. 
We note afterwards that, when $r \geq 4$, we may even take $\Delta(G) = 3$. The case where $r = 2$ remains open and seems interesting, and we leave this as an open question. 

If we parameterise by treewidth, then the parameterised complexity of $\MM(r,k)$ is much more reasonable.
For this, recall that $\XP$ denotes the class of \textit{slicewise polynomial} problems. That is, problems solvable in time $\mathcal O(f(k) \cdot |x|^{f(k)})$, where $k \in \N$ is some parameter and $f$ is some computable function. In particular, we observe the following. 

\begin{restatable}{alphtheorem}{xp}\label{thm:xp}
The problem $\MM(r,k)$ lies in $\XP$ when parameterised by the treewidth and maximum degree of the input graph. If $r \leq 3$, we need only parameterise by treewidth. 
\end{restatable}


This theorem follows fairly quickly from the main results of \cite{bonomo2022new}, which provide a general framework for solving `$r$-locally checkable' problems in graphs of bounded treewidth via dynamic programming.

\subsection*{Notational remarks}

Throughout this note, let $[n] := \{1, \ldots, n\}$. If $G$ is a (undirected) graph, we denote by $V(G)$ its vertex set and $E(G)$ its edge-set. Given  $u, v \in V(G)$, we denote by $\dist_G(u,v)$ the length of a shortest path in $G$ connecting $u$ to $v$, adopting the convention that $\dist_G(u,v) = \infty$ if no such path exists. Given $m \geq 0$, $v \in V(G)$, we let 
$
N_m(v) = \{u \in V(G) : \dist_G(u,v) \leq m\}
$
denote the closed $m$-neighbourhood of $v$ in $V(G)$. 
By $\Delta(G)$, we denote the maximum degree in the graph $G$.

\section{The reduction}

In this section, we prove our main result, i.e., that the problem $\MM(r,k)$ is $\NP$-com\-plete for every $r\geq 3$, $k\geq 2$, which proves Conjecture 9.6 from \cite{georgakopoulos2023graph}. The main idea of the proof is to give a polynomial-time reduction from $3\SAT$. Since the graph in our construction will be of degree at most four, we can even deduce $\NP$-completeness when the problem is restricted to graphs of maximum degree four. At the end of this section, we make a remark on how the maximum degree of the constructed graph can be reduced to three, whenever $r\geq 4$. 

We begin with the following easy lemma. 

\begin{lemma}\label{lem:reduce-to-2}
    Given $r\geq 1$ and $k \geq 2$, there is a polynomial-time reduction from $\MM(r,2)$ to $\MM(r,k)$. 
\end{lemma}

\begin{proof}
    Given an instance $(G, A, Z)$ of $\MM(r,2)$, we construct an instance of $\MM(r,k)$ by adding $k-2$ pairs of vertices $(a_i, z_i)$ to $G$, each pair connected by an edge. We add each $a_i$ to $A$ and each $z_i$ to $Z$. The resulting instance of $\MM(r,k)$ clearly has a solution if and only if the original instance of $\MM(r,2)$ does.
\end{proof}

We now recall the statement of the main theorem.



\npc*


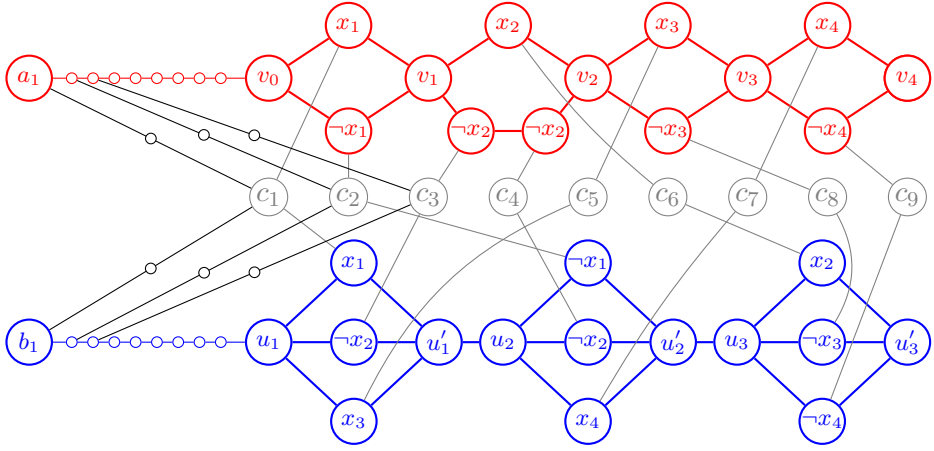
\begin{figure}
\centering

        
\def\minsize{17pt}
\def\minsizet{4pt}
\begin{tikzpicture}[scale=1.4]
\node (v0) at (0.75,0.5) [circle, red, thick, draw, inner sep=0pt, minimum size=\minsize] {\footnotesize $v_0$};
\node (v2) at (3.75,0.5) [circle, red, thick, draw, inner sep=0pt, minimum size=\minsize] {\footnotesize $v_2$};
\node (negx21) at (2.65,0) [circle, red, thick, draw, inner sep=0pt, minimum size=\minsize] {\footnotesize $\neg x_2$};
\node (negx22) at (3.35,0) [circle, red, thick, draw, inner sep=0pt, minimum size=\minsize] {\footnotesize $\neg x_2$};
\node (x2) at (3,1) [circle, red, thick, draw, inner sep=0pt, minimum size=\minsize] {\footnotesize $x_2$};
\foreach \i in {1,3,4} {
\node (negx\i) at (1.5*\i,0) [circle, red, thick, draw, inner sep=0pt, minimum size=\minsize] {\footnotesize $\neg x_\i$};
\node (x\i) at (1.5*\i,1) [circle, red, thick, draw, inner sep=0pt, minimum size=\minsize] {\footnotesize $x_\i$};
\node (v\i) at (1.5*\i+0.75,0.5) [circle, red, thick, draw, inner sep=0pt, minimum size=\minsize] {\footnotesize $v_\i$};
\draw[-, red, thick] (v\the\numexpr \i - 1 \relax) to (x\i);
\draw[-, red, thick] (v\the\numexpr \i - 1 \relax) to (negx\i);
\draw[-, red, thick] (v\i) to (x\i);
\draw[-, red, thick] (v\i) to (negx\i);
}
\draw[-, red, thick] (v1) to (negx21) to (negx22) to (v2) to (x2) to (v1);

\node (u1) at (0.75,-2) [circle, blue, thick, draw, inner sep=0pt, minimum size=\minsize] {\footnotesize $u_1$};
\node (u11) at (2.35,-2) [circle, blue, thick, draw, inner sep=0pt, minimum size=\minsize] {\footnotesize $u_1'$};
\node (u2) at (2.95,-2) [circle, blue, thick, draw, inner sep=0pt, minimum size=\minsize] {\footnotesize $u_2$};
\node (u21) at (4.55,-2) [circle, blue, thick, draw, inner sep=0pt, minimum size=\minsize] {\footnotesize $u_2'$};
\node (u3) at (5.15,-2) [circle, blue, thick, draw, inner sep=0pt, minimum size=\minsize] {\footnotesize $u_3$};
\node (u31) at (6.75,-2) [circle, blue, thick, draw,  inner sep=0pt, minimum size=\minsize, fill opacity=0.2, text opacity=1] {\footnotesize $u_3'$};

\node (L11) at (1.55,-1.25) [circle, blue, thick, draw, inner sep=0pt, minimum size=\minsize] {\footnotesize $x_1$};
\node (L12) at (1.55,-2) [circle, blue, thick, draw, inner sep=0pt, minimum size=\minsize] {\footnotesize $\neg x_2$};
\node (L13) at (1.55,-2.75) [circle, blue, thick, draw, inner sep=0pt, minimum size=\minsize] {\footnotesize $x_3$};
\draw[-,thick, blue] (u1) to (L11) to (u11) to (L12) to (u1) to (L13) to (u11) to (u2);

\node (L21) at (3.75,-1.25) [circle, blue, thick, draw, inner sep=0pt, minimum size=\minsize] {\footnotesize $\neg x_1$};
\node (L22) at (3.75,-2) [circle, blue, thick, draw, inner sep=0pt, minimum size=\minsize] {\footnotesize $\neg x_2$};
\node (L23) at (3.75,-2.75) [circle, blue, thick, draw, inner sep=0pt, minimum size=\minsize] {\footnotesize $x_4$};
\draw[-,thick, blue] (u2) to (L21) to (u21) to (L22) to (u2) to (L23) to (u21) to (u3);

\node (L31) at (5.95,-1.25) [circle, blue, thick, draw, inner sep=0pt, minimum size=\minsize] {\footnotesize $x_2$};
\node (L32) at (5.95,-2) [circle, blue, thick, draw, inner sep=0pt, minimum size=\minsize] {\footnotesize $\neg x_3$};
\node (L33) at (5.95,-2.75) [circle, blue, thick, draw, inner sep=0pt, minimum size=\minsize] {\footnotesize $\neg x_4$};
\draw[-,thick, blue] (u3) to (L31) to (u31) to (L32) to (u3) to (L33) to (u31);

\foreach \i in {1,...,9}{
\node (help\i) at (0.75*\i,-0.625) [circle, draw, gray, inner sep=1pt, minimum size=\minsizet, fill=white] {$c_\i$};
}
\draw[-, gray] (L11) to (help1) to (x1);
\draw[-, gray] (L21) to (help2) to (negx1);
\draw[-, gray] (L12) to (help3) to (negx21);
\draw[-, gray] (L22) to (help4) to (negx22);
\draw[-, gray] (L31) to (help6) to [bend left=7] (x2);
\draw[-, gray, bend left=20] (L13) to (help5); \draw[gray] (help5) to (x3);
\draw[-, gray] (L32) to [bend right=30](help8); \draw[gray] (help8) to (negx3);
\draw[-, gray, bend left=5] (L23) to (help7); \draw[gray] (help7) to (x4);
\draw[-, gray] (L33) to (help9) to (negx4);

\node (a0) at (-1.5,0.5) [circle, red, thick, draw, inner sep=0pt, minimum size=\minsize, fill=white, fill opacity=0.2, text opacity=1] {\footnotesize $a_1$};
\node (b0) at (-1.5,-2) [circle, blue, thick, draw, inner sep=0pt, minimum size=\minsize, fill opacity=0.2, text opacity=1] {\footnotesize $b_1$};

\node (v4) at (6.75,0.5) [circle, red, thick, draw, inner sep=0pt, minimum size=\minsize, fill=white, fill opacity=0.2, text opacity=1] {\footnotesize $v_4$};

\foreach \i in {1,...,8}{
\node (a\i) at (-1.3+0.2*\i, 0.5) [circle, draw, red, inner sep=0pt, minimum size=\minsizet, fill=white] {};
\node (b\i) at (-1.3+0.2*\i, -2) [circle, draw, blue, inner sep=0pt, minimum size=\minsizet, fill=white] {};
\draw[-,red] (a\i) to (a\the\numexpr \i - 1 \relax);
\draw[-,blue] (b\i) to (b\the\numexpr \i - 1 \relax);
}
\draw[-,red] (a8) to (v0);
\draw[-,blue] (b8) to (u1);
\draw[-] (a0) -- (help1) node[midway, circle, draw, inner sep=0pt, minimum size=\minsizet, fill=white] {};
\draw[-] (b0) -- (help1) node[midway, circle, draw, inner sep=0pt, minimum size=\minsizet, fill=white] {};
\draw[-] (a1) -- (help2) node[midway, circle, draw, inner sep=0pt, minimum size=\minsizet, fill=white] {};
\draw[-] (b1) -- (help2) node[midway, circle, draw, inner sep=0pt, minimum size=\minsizet, fill=white] {};
\draw[-] (a2) -- (help3) node[midway, circle, draw, inner sep=0pt, minimum size=\minsizet, fill=white] {};
\draw[-] (b2) -- (help3) node[midway, circle, draw, inner sep=0pt, minimum size=\minsizet, fill=white] {};
\end{tikzpicture}

\caption{An example of the reduction from $3\SAT$ to $\MM(3,2)$. The graph constructed above encodes the formula $(x_1 \vee \neg x_2 \vee x_3) \wedge (\neg x_1 \vee \neg x_2 \vee x_4) \wedge (x_2 \vee \neg x_3 \vee \neg x_4)$. For better overview, the dummy paths beginning in vertices $a_4, \dots, a_9, b_4, \dots, b_9$ are omitted.}
\label{fig:example-reduction}
\end{figure}

\begin{proof}
First, note that, given an instance of $\MM(r,k)$ and $k$ paths, one can check in polynomial time whether the paths form a feasible solution. Therefore, the problem  $\MM(r,k)$ is in $\NP$ for every choice of $r$ and $k$. 

We now turn to proving $\NP$-hardness by giving a polynomial-time transformation from $3\SAT$. By Lemma~\ref{lem:reduce-to-2}, it suffices to show hardness for $k=2$. Therefore, we give a reduction from $3\SAT$ to $\MM(r,2)$ for an arbitrary but fixed $r\geq 3$.

Consider an arbitrary instance of $3\SAT$ over Boolean variables $x_1, \ldots, x_n$ given by
\begin{equation*}
    \varphi = \bigwedge_{j = 1}^m (L_{j,1}\vee L_{j,2} \vee L_{j,3}), 
\end{equation*}
where each $L_{j,i}$ is a literal, i.e., $L_{j,i}\in \bigcup_{l=1}^n \{x_l, \neg x_l\}$.
The goal is to construct an instance of $\MM(r,2)$ for which the answer is `Yes' if and only if $\varphi$ is satisfiable.
We construct two graphs $G_1$ and $G_2$, before connecting them together with various paths to form a graph~$G$ on which our final instance for $\MM(r,2)$ is defined. 
    
First, we construct $G_1$ as follows (illustrated by the red subgraph in Figure~\ref{fig:example-reduction}):
We begin with introducing vertices labelled $v_0, \dots, v_n$.
For every $i \in [n]$, let $\alpha_0(i)$ denote the number of times that the literal $\neg x_i$ appears in $\varphi$ and let $\alpha_1(i)$ denote the number of times that the literal $x_i$ appears in $\varphi$. We connect $v_{i-1}$ to $v_i$ by two paths, one of $\alpha_0(i)+1$ edges, denoted $P_{i,0}$, and one of $\alpha_1(i)+1$ edges, denoted $P_{i,1}$ ($i\in[n])$. 
Last, we attach a path of $3m$ edges to $v_0$ and label the vertices in this path $a_1, \dots, a_{3m}$, where $a_1$ is the start vertex of this path.
Note that $G_1$ consists of $3m+(n+1)+m=\bigO(n+m)$ vertices. 

Intuitively, the graph $G_1$ plays the following role in our final construction: We will force one of the paths to never leave $G_1$ and to connect the vertices $a_1$ and $v_n$. On its way, it has to use exactly one of the paths $P_{i,0}, P_{i,1}$ for every $i \in [n]$. If it uses path $P_{i,b}$, this will correspond to an assignment of the variables of the SAT-instance where $x_i=\neg b$ ($b \in \{0,1\}$).

Next, we construct $G_2$ as follows (illustrated by the blue subgraph in Figure~\ref{fig:example-reduction}):
For every $i \in [m]$, we introduce vertices labelled $u_i, u_i', L_{i,1},  L_{i,2}, L_{i,3}$. We connect each $L_{i,j}$ to $u_i$ and to $u_i'$ and, for $i<m$, we connect $u_i'$ to $u_{i+1}$. Similarly as in the construction of~$G_1$, we attach a path of $3m$ edges to $u_1$ and label the vertices $b_1, \dots, b_{3m}$, where $b_1$ is the start vertex of this path.
Note that $G_2$ contains $3m+5m=\bigO(m)$ vertices. 

The intuitive role of $G_2$ in the final construction will be as follows: One of the paths will be forced to connect $b_1$ to $u_m'$ and never leave $G_2$. On its way, it has to visit one of the vertices labelled $L_{i,1},  L_{i,2}, L_{i,3}$ for every $i \in [m]$. If it visits, say $L_{i,1}$, this will correspond to the literal $L_{i,1}$ evaluating to 1 in the assignment given by the path in $G_1$. The existence of a path from $b_1$ to $u_m'$ then means that, in each clause, one of the literals evaluates to~1, i.e., the SAT-formula is satisfied.


    Next up, we connect these two graphs together. 
    Recall that the number of vertices in the path $P_{i,0}$ is the number of times that the literal $\neg x_i$ appears in $\varphi$. We connect each vertex in $P_{i,0}$ to a distinct vertex labelled $L_{j,l}\in G_2$ where $L_{j,l}=\neg x_i$ by a path or $r-1$ edges. Similarly, we connect the vertices in $P_{i,1}$ to vertices in $G_2$ labelled by literals that are $x_i$ by a path of length $r-1$ (illustrated by the gray paths in Figure~\ref{fig:example-reduction}).
    We call such a path between $G_1$ and $G_2$ an \emph{exclusion path} (because, since it has length $r-1$, in a solution to $\MM(r,2)$, it is not possible that both of it end vertices are visited by different paths). 

    Note that we have added precisely $3m$ exclusion paths and, since we assumed $r\geq 3$, every one of them contains a vertex. Choose one vertex on every exclusion path and label them by $c_1, \dots, c_{3m}$. The last step of our construction is to add paths of length $r-1$ from $a_i$ to $c_i$ and from $b_i$ to $c_i$  for all $i \in [3m]$ (illustrated by the black paths in Figure~\ref{fig:example-reduction}). We call these \emph{dummy paths.} 

    We define our final instance of $\MM(r,2)$ now as follows: We consider the graph~$G$ that consists of $G_1, G_2$, the exclusion paths and the dummy paths. The two sets that need to be connected by two paths are given by $A=\{a_1, b_1\}$ and $Z=\{v_n, u_m'\}$.
    We have already seen that $|G_1|+|G_2|= \bigO(n+m)$. Since every vertex in $G_1$ and $G_2$ is adjacent to at most one dummy path or extension path and these have length at most $r$, we obtain that $|G|= \bigO(rn+rm)$, which is polynomial. Also, note that the maximum degree of $G$ is four.

    Now, we turn to proving that there is a feasible solution to the defined instance of $\MM(r,2)$ if and only if there is a satisfying assignment to $\varphi$.
    

    First, suppose that $\varphi$ is satisfiable, so there is some assignment $f \colon  [n] \to \{0,1\}$ which satisfies $\varphi$. Consider the path $P$ from $a_1$ to $v_n$ defined by the sequence
    $$
    P = a_1, \dots, a_{3m}, v_0, P_{1, \neg f(1)}, v_1 , P_{2, \neg f(2)}, \ldots ,v_{n-1},  P_{n, \neg f(n)}, v_n.
    $$
    Now, for each $j \in [m],$ we have that at least one of $L_{j,1}, L_{j,2}, L_{j,3}$ evaluates to $1$ in the satisfying assignment $f$. Let $g(j)\in \{1,2,3\}$ be chosen such that each $L_{j,g(j)}$ evaluates to 1.
    Consider the path 
    $$
    P' = b_1, \dots, b_{3m}, u_1, L_{1,g(1)} , u_1', u_2, L_{2,g(2)}, u_2', \ldots, u_{m}, L_{m,g(m)}, u_m'. 
    $$
    It is left to argue that $P$ and $P'$ are $(r-1)$-disjoint. For this, note that $P\subseteq G_1$, $P'\subseteq G_2$ and the only pairs of vertices in $G_1, G_2$ that are at distance less than $r$ from each other are the pairs labelled $(L_{i,j}, x_l)$ if $L_{i,j}=x_l$ or $(L_{i,j}, \neg x_l)$ if $L_{i,j}= \neg x_l$. By construction, path $P'$ only uses $L_{i,j}$ if it evaluates to 1. Wlog.~assume $L_{i,j}=x_l$. Then $x_l$ is set to true in the assignment $f$, which implies that $P$ does not use the path $P_{l,1}$. Therefore, the paths are $(r-1)$-disjoint and form a feasible solution to this instance of $\MM(r,2)$.

    
    Conversely, suppose we have a solution of (simple) paths $(P,P')$ to the given instance of $\MM(r,2)$, where $P$ begins in $a_1$, and $P'$ begins in $b_1$. 
    As a first step, we argue that the head of $P$ is $a_1, \dots, a_{3m}, v_0$ and the head of $P'$ is $b_1, \dots, b_{3m}, u_1$. For this, note that the second vertex of $P$ is either $a_2$ or $P$ follows the dummy path to $c_1$. However, $P$ cannot visit $c_1$ because this vertex is at distance $r-1$ from $b_1$, which is in $P'$. Therefore, the second vertex of $P$ is $a_2$ and, similarly, the second vertex of $P'$ is $b_2$. It follows inductively that the head of $P$ is $a_1, \dots, a_{3m}, v_0$ and the head of $P'$ is $b_1, \dots, b_{3m}, u_1$.
    
    Note that any simple path starting in a vertex of $G_1$, leaving $G_1$ and ending in a vertex of $G_1 \cup G_2$ contains one of the vertices $c_i$. As these are at distance $r-1$ from $a_i$ and $b_i$ and we have already seen that $a_i\in P$ and $b_i \in P'$, we obtain that $P\subseteq G_1$ and $P' \subseteq G_2$. In particular, the end vertex of $P$ is $v_n$ and the end vertex of $P'$ is $u_m'$.
    It follows that $P$ contains exactly one of the paths $P_{i,0}, P_{i,1}$ for every $i \in [n]$. From this, we construct an assignment of the variables $x_1, \dots, x_n$ by setting $x_i=1$ if and only if $P$ contains $P_{i,0}$. 

    We argue that this assignment is satisfying. For this, note that the path $P'$ visits exactly one of the $L_{j,1}, L_{j,2}, L_{j,3}$ for every $j \in [m]$. We claim that each visited literal evaluates to 1 in the constructed assignment. Assume wlog.~that $L_{j,i}=x_l$. By construction, the distance between the vertex labelled $L_{j,i}$ and the path $P_{l,1}$ is $r-1$ so that $P'$ can only visit $L_{j,i}$ if $P$ does not visit $P_{l,1}$. This means that the variable $x_l$ is set to 1 so that the literal $L_{j,i}$ indeed evaluates to 1.

    This completes the proof of the fact that the answer to the constructed instance of $\MM(r,2)$ is `Yes' if and only if $\varphi$ has a satisfying assignment. Therefore, $\MM(r,2)$ is $\NP$-complete. Since our constructed graph has maximum degree four, we even obtain $\NP$-completeness when the problem is restricted to graphs of maximum degree four.
    %
    %
    %
    %
\end{proof}

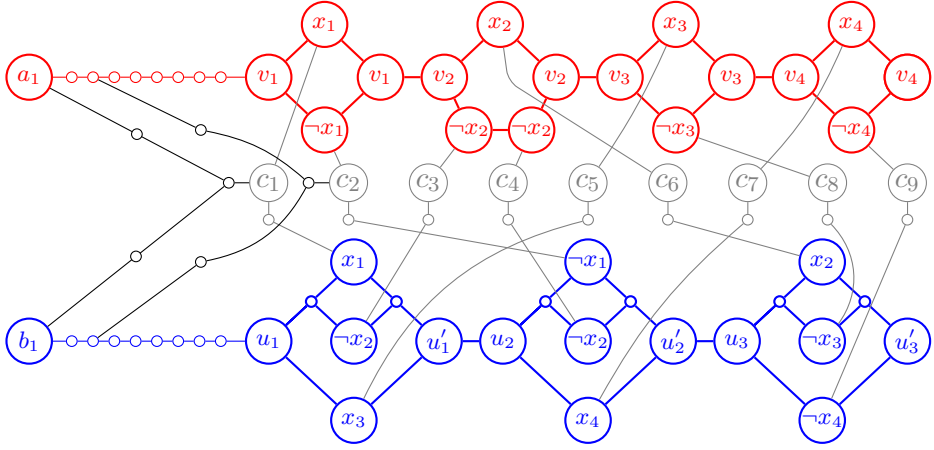
\begin{figure}
    \centering
\def\minsize{17pt}
\def\minsizet{4pt}
\begin{tikzpicture}[scale=1.4]
\node (v21) at (3.45,0.5) [circle, red, thick, draw, inner sep=0pt, minimum size=\minsize] {\footnotesize $v_2$};
\node (v2) at (2.4,0.5) [circle, red, thick, draw, inner sep=0pt, minimum size=\minsize] {\footnotesize $v_2$};
\node (negx21) at (2.63,0) [circle, red, thick, draw, inner sep=0pt, minimum size=\minsize] {\footnotesize $\neg x_2$};
\node (negx22) at (3.22,0) [circle, red, thick, draw, inner sep=0pt, minimum size=\minsize] {\footnotesize $\neg x_2$};
\node (x2) at (2.925,1) [circle, red, thick, draw, inner sep=0pt, minimum size=\minsize] {\footnotesize $x_2$};

\foreach \i in {1,3,4} {
\node (negx\i) at (0.75+\i*1.65-1.65+0.525,0) [circle, red, thick, draw, inner sep=0pt, minimum size=\minsize] {\footnotesize $\neg x_\i$};
\node (x\i) at (0.75+\i*1.65-1.65+0.525,1) [circle, red, thick, draw, inner sep=0pt, minimum size=\minsize] {\footnotesize $x_\i$};
\node (v\i) at (0.75+\i*1.65-1.65,0.5) [circle, red, thick, draw, inner sep=0pt, minimum size=\minsize] {\footnotesize $v_\i$};
\node (v\i1) at (0.75+\i*1.65-1.65+1.05,0.5) [circle, red, thick, draw, inner sep=0pt, minimum size=\minsize] {\footnotesize $v_\i$};
\draw[-, red, thick] (v\i) to (x\i) to (v\i1);
\draw[-, red, thick] (v\i) to (negx\i) to (v\i1);
}

\draw[-, red, thick] (v11) to (v2) to (negx21) to (negx22) to (v21) to (v3);
\draw[-, red, thick] (v2) to (x2) to (v21);
\draw[red, thick] (v31) to (v4);

\node (u1) at (0.75,-2) [circle, blue, thick, draw, inner sep=0pt, minimum size=\minsize] {\footnotesize $u_1$};
\node (u11) at (2.35,-2) [circle, blue, thick, draw, inner sep=0pt, minimum size=\minsize] {\footnotesize $u_1'$};
\node (u1h2) at (1.95,-1.625) [circle, blue, thick, draw, inner sep=0pt, minimum size=\minsizet] {};
\node (u1h1) at (1.15,-1.625) [circle, blue, thick, draw, inner sep=0pt, minimum size=\minsizet] {};
\node (u2) at (2.95,-2) [circle, blue, thick, draw, inner sep=0pt, minimum size=\minsize] {\footnotesize $u_2$};
\node (u21) at (4.55,-2) [circle, blue, thick, draw, inner sep=0pt, minimum size=\minsize] {\footnotesize $u_2'$};
\node (u2h1) at (3.35,-1.625) [circle, blue, thick, draw, inner sep=0pt, minimum size=\minsizet] {};
\node (u2h2) at (4.15,-1.625) [circle, blue, thick, draw, inner sep=0pt, minimum size=\minsizet] {};
\node (u3) at (5.15,-2) [circle, blue, thick, draw, inner sep=0pt, minimum size=\minsize] {\footnotesize $u_3$};
\node (u31) at (6.75,-2) [circle, blue, thick, draw, inner sep=0pt, minimum size=\minsize, fill opacity=0.2, text opacity=1] {\footnotesize $u_3'$};
\node (u3h1) at (5.55,-1.625) [circle, blue, thick, draw, inner sep=0pt, minimum size=\minsizet] {};
\node (u3h2) at (6.35,-1.625) [circle, blue, thick, draw, inner sep=0pt, minimum size=\minsizet] {};

\node (L11) at (1.55,-1.25) [circle, blue, thick, draw, inner sep=0pt, minimum size=\minsize] {\footnotesize $x_1$};
\node (L12) at (1.55,-2) [circle, blue, thick, draw, inner sep=0pt, minimum size=\minsize] {\footnotesize $\neg x_2$};
\node (L13) at (1.55,-2.75) [circle, blue, thick, draw, inner sep=0pt, minimum size=\minsize] {\footnotesize $x_3$};
\draw[-,thick, blue] (u1) to (u1h1) to (L11) to (u1h2) to (L12) to (u1h1) to (u1) to (L13) to (u11) to (u1h2);
\draw[-,thick, blue] (u11) to (u2);

\node (L21) at (3.75,-1.25) [circle, blue, thick, draw, inner sep=0pt, minimum size=\minsize] {\footnotesize $\neg x_1$};
\node (L22) at (3.75,-2) [circle, blue, thick, draw, inner sep=0pt, minimum size=\minsize] {\footnotesize $\neg x_2$};
\node (L23) at (3.75,-2.75) [circle, blue, thick, draw, inner sep=0pt, minimum size=\minsize] {\footnotesize $x_4$};
\draw[-,thick, blue] (u2) to (u2h1) to (L21) to (u2h2) to (L22) to (u2h1) to (u2) to (L23) to (u21) to (u2h2);
\draw[-,thick, blue] (u21) to (u3);

\node (L31) at (5.95,-1.25) [circle, blue, thick, draw, inner sep=0pt, minimum size=\minsize] {\footnotesize $x_2$};
\node (L32) at (5.95,-2) [circle, blue, thick, draw, inner sep=0pt, minimum size=\minsize] {\footnotesize $\neg x_3$};
\node (L33) at (5.95,-2.75) [circle, blue, thick, draw, inner sep=0pt, minimum size=\minsize] {\footnotesize $\neg x_4$};
\draw[-,thick, blue] (u3) to (u3h1) to (L31) to (u3h2) to (L32) to (u3h1) to (u3) to (L33) to (u31) to (u3h2);

\foreach \i in {1,...,9}{
\node (help\i) at (0.75*\i,-0.5) [circle, draw, gray, inner sep=1pt, fill=white] {$c_\i$};
\node (help\i1) at (0.75*\i,-0.85) [circle, draw, gray, inner sep=1pt, minimum size=\minsizet, fill=white] {};
}

\draw[-, gray] (L11) to (help11) to (help1) to (x1);
\draw[-, gray] (L21) to (help21) to (help2) to (negx1);
\draw[-, gray] (L12) to (help31) to (help3) to (negx21);
\draw[-, gray] (L22) to (help41) to (help4) to (negx22);
\draw[-, gray] (L31) to (help61) to (help6); \draw[gray, rounded corners] (help6) to (3,0.46) to (x2);
\draw[-, gray, bend left=25] (L13) to (help51); \draw[gray] (help51) to (help5) to [bend right=5] (x3);
\draw[-, gray] (L32) to [bend right=40](help81); \draw[gray] (help81) to (help8) to (negx3);
\draw[-, gray, bend left=10] (L23) to (help71); \draw[gray] (help71) to (help7) to [bend right=10] (x4);
\draw[-, gray] (L33) to (help91) to (help9) to (negx4);

\node (a0) at (-1.5,0.5) [circle, red, thick, draw, inner sep=0pt, minimum size=\minsize, fill=white, fill opacity=0.2, text opacity=1] {\footnotesize $a_1$};
\node (b0) at (-1.5,-2) [circle, blue, thick, draw,  inner sep=0pt, minimum size=\minsize, fill opacity=0.2, text opacity=1] {\footnotesize $b_1$};

\node (v4) at (6.75,0.5) [circle, red, thick, draw, inner sep=0pt, minimum size=\minsize, fill=white, fill opacity=0.2, text opacity=1] {\footnotesize $v_4$};

\foreach \i in {1,...,8}{
\node (a\i) at (-1.3+0.2*\i, 0.5) [circle, draw, red, inner sep=0pt, minimum size=\minsizet, fill=white] {};
\node (b\i) at (-1.3+0.2*\i, -2) [circle, draw, blue, inner sep=0pt, minimum size=\minsizet, fill=white] {};
\draw[-,red] (a\i) to (a\the\numexpr \i - 1 \relax);
\draw[-,blue] (b\i) to (b\the\numexpr \i - 1 \relax);
}
\draw[-,red] (a8) to (v1);
\draw[-,blue] (b8) to (u1);
\node (c11) at (0.375, -0.5) [circle, draw, inner sep=0pt, minimum size=\minsizet] {};
\node (c21) at (1.125, -0.5) [circle, draw, inner sep=0pt, minimum size=\minsizet] {};
\draw[-] (a0) -- (c11) node[midway, circle, draw, inner sep=0pt, minimum size=\minsizet, fill=white] {};
\draw[-] (b0) -- (c11) node[midway, circle, draw, inner sep=0pt, minimum size=\minsizet, fill=white] {};
\draw[-] (c11) to (help1);
\node (dummy1) at (0.1125,0) [circle, draw, inner sep=0pt, minimum size=\minsizet, fill=white] {};
\node (dummy2) at (0.1125,-1.25) [circle, draw, inner sep=0pt, minimum size=\minsizet, fill=white] {};
\draw[-] (a2) to (dummy1) to [bend left=15] (c21);
\draw[-] (b2) to (dummy2) to [bend right=25] (c21);
\draw[-] (c21) to (help2);
\end{tikzpicture}
\caption{An example of the transformation from $3\SAT$ to $\MM(4,2)$ with an input graph of maximum degree three.}
\label{fig:reduction-degree-three}
\end{figure}

To complete this section, we observe that, in the case $r\geq 4$, one can adjust the construction in the proof of Theorem~\ref{thm:np-complete} such that the resulting graph has degree at most three, as illustrated in Figure~\ref{fig:reduction-degree-three}. Note that, we can apply the same ideas to reduce the degrees of $G_1$ and $G_2$ in our proof, but it is not obvious how to reduce the degree of the vertices $c_1, \dots, c_{3m}$ in the case $r\leq 3$. More precisely, if $r\geq 4$, we can make the dummy paths beginning in $a_i$ and $b_i$ share their last edge to reduce the degree of $c_i$. This is not possible in the case $r= 3$ because then $a_i$ and $b_i$ would be at distance 2 so that it is immediate that no feasible solution exists. In other words, the dummy paths are only allowed to share their last edge if $2(r-2)\geq r$, which requires $r\geq 4$.

\begin{observation}
For every $r\geq 4$ and $k\geq 2$, the problem $\MM(r,k)$ is $\NP$-complete, even on graphs of degree at most three.
\end{observation}



\section{Graphs of bounded treewidth}

To contrast the above, we now show that $\MM(r,k)$ is an $r$-locally checkable problem, in the sense of \cite{bonomo2022new}. This implies the existence of a polynomial time algorithm if the input graph is of bounded treewidth and bounded degree. 

We now state a definition of local checkability. This will be more restrictive than that of \cite{bonomo2022new}, but is sufficient for our purposes. For the full definition, see \cite[§~3]{bonomo2022new}. 

\begin{definition}[Simple $m$-locally checkable problem]
    Given $m \geq 1$, a \textit{simple $m$-locally checkable problem} is a problem of the following form: 

    \textit{Instance:}
    \begin{itemize}
        \item A finite graph $G$,
        \item a set $L$ of colours,
        \item subsets $L_v \subset L$ of allowable colours for every $v \in V(G)$,
        \item a polynomial-time-computable function $\chck(v,c)$, which takes in a vertex $v \in V(G)$ and an assignment $c$ of $N_m(v)$, and outputs either $0$ or $1$ depending on whether a certain condition is satisfied, where an \emph{assignment} is a map
        $
        c \colon U \to L
        $ 
        such that $c(u) \in L_u$ for every $u \in U$. 
    \end{itemize}
    
    
    \textit{Question:} Does there exist an assignment $c \colon V(G) \to L$  such that 
    $$
    \chck(v,c|_{N_m(v)}) = 1
    $$
    for all $v \in V(G)$?
\end{definition}

The main results of \cite{bonomo2022new} imply the following. 

\begin{theorem}[{\cite[Thm.~5.5.1, Cor.~6.0.3]{bonomo2022new}}]\label{thm:locally-check}
    Let $m \geq 1$.
    When parameterised by treewidth and maximum degree, every (simple) $m$-locally checkable problem is in $\XP$. If $m = 1$, we need only parameterise by treewidth. 
\end{theorem}

In order to make efficient use of this theorem, we introduce the following variant of the metric-Menger decision problem. 

       

\begin{problem*}
        Given $r, k \geq 1$, define $\MM_\mathrm{p}(r, k)$ as follows:
        
            \textit{Instance:} Finite graph $G$, and $k$ pairs of terminals $(s_1, t_1), \ldots, (s_k, t_k)$ in $V(G)$.
            
            \textit{Question:} Do there exist $k$ pairwise $(r-1)$-disjoint paths $P_1, \ldots P_k$, such that $P_i$ connects $s_i$ to $t_i$?
    \end{problem*}

Note that, if we set $r = 1$, then we recover precisely the classical disjoint paths problem.  
One can also note the following relationship between the two variants of the metric Menger problem.

\begin{lemma}\label{lem:mmp-to-mm}
    There is a polynomial-time truth table reduction\footnote{Recall that, in a \emph{truth table reduction}, we construct multiple but polynomially many instances of the problem we reduce to, such that the output of the original problem can be expressed as a polynomial-time computable function of the outputs of the constructed instances. By contrast, in a transformation, one needs to construct a single instance of the problem we reduce to, to which the answer is `Yes' if and only if the answer to the original instance is `Yes'.} from $\MM(r,k)$ to $\MM_\mathrm{p}(r,k)$. Given an instance of $\MM(r,k)$, the constructed instances of $\MM_\mathrm{p}(r,k)$ have the same input graph as the original problem. 
\end{lemma}

\begin{proof}
    Consider an instance of $\MM(r,k)$ with input graph $G$ and $A, Z \subset V(G)$. 
    There are $\mathcal O(n^{2k})$ ways to choose $k$ pairs of terminals $(s_1, t_1), \ldots, (s_k, t_k)$, where each $s_i \in A$, $t_i \in Z$. Each such choice determines an instance of $\MM_\mathrm{p}(r,k)$. The original instance of $\MM(r,k)$ has a solution if and only if at least one of the constructed instances of $\MM_\mathrm{p}(r,k)$ has a solution. 
\end{proof}

\begin{lemma}\label{lem:local-check}
    There is a polynomial-time transformation from $\MM_\mathrm{p}(r,k)$ to a simple $\lfloor \tfrac r 2\rfloor$-locally checkable problem on the same input graph. 
\end{lemma}

\begin{proof}
    Let $G$ be the input graph and
    let $\{(s_1, t_1), \ldots, (s_k, t_k)\}$ denote the input terminals. To ease notation, let $m = \lfloor \tfrac r 2 \rfloor$. 

    We construct an instance of an $m$-locally checkable problem which is equivalent to $\MM_\mathrm{p}(r,k)$.
    The set of possible colours $L$ is taken to be 
    $$
    L = ([k] \times \{0,1,2\}) \cup \{\bot\}. 
    $$
    Let $v \in V(G)$. We now define $L_v \subset L$. If $v = s_i = t_i$ for some $i \in [k]$, then set 
    $
    L_v = \{(i, 0)\}
    $. If there is some $i \in [k]$ such that $s_i \neq t_i$ and $v \in \{s_i, t_i\}$, then set $L_v = \{(i, 1)\}$. 
    Otherwise we set 
    $$
    L_v = ([k] \times \{2\}) \cup \{\bot\}.
    $$
    Given an assignment $c$, we say that a vertex $u$ is \textit{coloured} if $c(u) \neq \bot$. 
    
    The function $\chck$ is now defined as follows. Given $v \in V(G)$ and an assignment $c : N_m(v) \to L$, we set $\chck(v,c) = 1$ if and only if both of the following hold: 
    \begin{itemize}
        \item For all coloured $u, u' \in N_m(v)$, say 
        $c(u) = (i,j)$ and $c(u') = (i',j')$, 
        we have that $i = i'$, and 
        
        \item if $c(v) = (i,j) \neq \bot$, then exactly $j$ vertices adjacent to $v$ are coloured. 
    \end{itemize}
    This concludes the construction. It is clear this instance can be built in polynomial time. 

    We now show that the two instances are equivalent. 
    Suppose there exists a set of pairwise $(r-1)$-disjoint paths $P_1, \ldots , P_k$ where $P_i$ connects $s_i$ to $t_i$. We construct an assignment $c \colon V(G) \to L$ to the instance above via:
    $$
    c(v) = 
    \begin{cases}
        (i, 2) &\text{if $v \in V(P_i)$, $v \neq s_i$ and $v \neq t_i$, }\\
        (i, 1) &\text{if $s_i \neq t_i$, and $v \in \{s_i, t_i\}$,  }\\
        (i, 0) &\text{if $v = s_i = t_i$, }\\
        \bot &\text{otherwise.}
    \end{cases}
    $$
    It is clear that $c$ is well-defined, and since the $P_i$ are pairwise $(r-1)$-disjoint, we certainly have that $c$ is a valid solution to the $m$-locally checkable problem constructed above. 
    
    Conversely, suppose $c \colon  V(G) \to L$ is a valid assignment which solves the above problem. For each $i \in [k]$, let 
    $$
    Q_i = \{v \in V(G) : \text{$c(v) = (i,j)$ for some $j \in \{0,1,2\}$}\}.
    $$
    Firstly, we claim that every point in $Q_i$ lies at distance at least $r$ from $Q_j$, for $j \neq i$. Indeed, suppose not, and let $P$ be a path of length at most $r-1$ connecting $Q_i$ to $Q_j$. Let $v_0, \ldots v_{r-1}$ be the sequence of vertices appearing along $P$. Since $2m \geq r-1$, we have in any case that both 
    $$
    N_m(v_m) \cap Q_i \neq \emptyset \ \ \text{and} \ \  N_m(v_m) \cap Q_j \neq \emptyset.
    $$
    Thus, $
    \chck(v_m,c|_{N_m(v_m)}) = 0
    $ by definition, contradicting the assumption that $c$ is a valid assignment. 
    Now, each $Q_i$ is a (possibly disconnected) subgraph where every vertex has degree 2, except for $s_i$ and $t_i$ which have degree 1 if they are distinct and degree 0 otherwise. In particular, $Q_i$ is forced to contain a path connecting $s_i$ to $t_i$. Let $P_i$ denote this path.  Repeating this for every $i$, we see that the resulting paths $P_1, \ldots, P_k$ form a solution to $\MM_\mathrm{p}(r,k)$. 
\end{proof}

If we combine Theorem~\ref{thm:locally-check} with Lemmas~\ref{lem:mmp-to-mm} and \ref{lem:local-check}, we immediately deduce our second theorem. Note that it is important that, in Lemmas~\ref{lem:mmp-to-mm} and \ref{lem:local-check}, the input graph does not change in the reductions. This means that the properties of bounded treewidth and bounded degree are preserved by the reductions. 

\xp*

\section{Open questions}

As discussed in the introduction, it is known that $\MM(1,k)$ is solvable in polynomial time for every $k \geq 1$. The construction employed in the proof of Theorem~\ref{thm:np-complete} in the case where $r \geq 3$ does not seem to adapt well to the $r = 2$ case. This suggests the following question.

\begin{question}
    Given $k \geq 2$, how hard is $\MM(2, k)$? Is it a good candidate for an $\NP$-intermediate problem?
\end{question}

Recall that a decision problem in $\NP$ is said to be \textit{$\NP$-intermediate} if it does not lie in~$\P$, nor is it $\NP$-complete. It is known that $\NP$-intermediate problems exist if and only if $\P \neq \NP$.

\medskip

We have seen (cf. Theorem~\ref{thm:xp}) that $\MM(r,k)$ is solvable in polynomial time for graphs of bounded treewidth and bounded degree. It is therefore natural to ask if the same is true for other structural restrictions. 

\begin{question}
    If we restrict the input to planar graphs, can we solve $\MM(n,k)$ efficiently? What about minor excluded graphs?
\end{question}

\subsection*{Acknowledgements}

The authors thank Agelos Georgakopoulos for a careful reading of an old version of this paper.

\bibliographystyle{abbrv}
\bibliography{references}

\end{document}